\documentclass[12pt,leqno,twoside,a4paper]{article}
\usepackage{amsfonts,amsmath,amsthm,amssymb}
\usepackage[english]{babel}
\newtheorem{theorem}{Theorem}

\newtheorem{lemma}[theorem]{Lemma}

\theoremstyle{definition}

\theoremstyle{remark}
\newtheorem{remark}[theorem]{Remark}
\theoremstyle{remarks}

 \numberwithin{equation}{section}

\newcommand{\de}{\delta}

\newcommand{\T}[1]{{#1}[x;\si ,\de]}
\newcommand{\Tn}{R[x;\si ,\de]}

\newcommand{\A}[1]{{#1}[x;\si]}
\newcommand{\si}{\sigma}

\begin{document}

\title{\normalsize \sc  On $q$-skew Iterated Ore Extensions Satisfying  a Polynomial Identity}
\author{ {\bf \normalsize Andr\'{e} Leroy and Jerzy
Matczuk\footnote{ The research was
 supported by Polish MNiSW grant No. N N201 268435}} \vspace{6pt}\\
 \normalsize  Universit\'{e} d'Artois,  Facult\'{e} Jean Perrin\\
\normalsize Rue Jean Souvraz  62 307 Lens, France\\
   \normalsize  E-mail: leroy@euler.univ-artois.fr \vspace{6pt}\\
 \normalsize Institute of Mathematics, Warsaw University,\\
\normalsize Banacha 2, 02-097 Warsaw, Poland\\
 \normalsize E-mail: jmatczuk@mimuw.edu.pl}
\date{ }
\maketitle\markboth{\rm  A. LEROY AND J.MATCZUK}{ \rm PI $q$-SKEW
ORE EXTENSIONS}


\begin{abstract}
For iterated Ore extensions satisfying a polynomial identity
we present an elementary way of erasing derivations.  As a consequence we recover
some results obtained by Haynal in \cite{H}.  We also prove, under mild assumptions on $R_n=R[x_1;\si_1,\de_1]\dots[x_n;\si_n;\de_n]$ that the Ore extension $R[x_1;\si_1]\dots[x_n;\si_n]$ exists and is PI if $R_n$ is PI.
\end{abstract}

For an Ore extension $\T{R}$, where
 $\si$ is an injective endomorphism of a  prime ring $R$ and
 $\de$ is a  $\si$-derivation of $R$, necessary and
 sufficient conditions for being a PI ring can be found in
  \cite{LM}.  The main result of \cite{H} states that if $R$ is a noetherian domain which is also an algebra over a field $k$ then,
under some quantum like hypotheses, the
iterated Ore extensions $R_n=R[x_1;\si_1,\de_1]\dots
[x_n;\si_n,\de_n]$ and $T_n=R[x_1;\si_1]\dots [x_n;\si_n]$ have the
same PI degree.
This result was known earlier only in the case
$\mbox{char}\,  k=0$ (Cf. \cite{Jon}). Haynal achieved the above
mentioned  result by generalizing   Cauchon's erasing of
derivations procedure (Cf. \cite{C}).

 The aim of this paper is to present a short and elementary
proof of an erasing of derivations process under the hypotheses that
the iterated Ore extension is prime PI.  Using this method we can recover some of Haynal's results in a slightly more general setting (Cf. Theorems \ref{Iterated Ore extensions} and   \ref{Haynal}).
We believe  that our approach explains also  the
nature of some of the assumptions.

Notice that for a general Ore extension
$R_n$ the extension $T_n$ mentioned above does not have to exist; for example
the automorphism $\si_2$ of $R_1$ has, in general, no meaning as a map of
$T_1$.  We show in Theorem \ref{graded method} that, under mild assumptions on $R_n$, $T_n$ does exist.
Moreover it satisfies a polynomial identity provided $R_n$ does.

Throughout the paper $R$ stands for an associative ring with unity
and $Z(R)$  for its center.  For an automorphism $\si$ and a
$\si$-derivation $\de$ of $R$, $\T{R}$ denotes  the
Ore extension with coefficients written on the
left.   We denote $R[x;\si]$ for $R[x;\si, 0]$ and $R[x;\de]$ for $R[x;id,\de]$.

Let $T$ be a subring of a ring $W$ and $\si,\,\de$ be an automorphism and a
$\si$-derivation of $T$, respectively.  We will write $T[y;\si,\de]\subseteq W$
for some $y\in W$, if  the set $\{y^n\}_{n= 0}^\infty$ is left $T$-independent
and $yr=\si(r)y +\de(r)$, for any $r\in T$.  That is the subring of $W$
generated by $T$ and $y$ forms an Ore extension.

Let $q\in Z(R)$ be such that $\si(q)=q$ and $\de(q)=0$.
We  say
that $\de$ is a $q$-skew $\si$-derivation of $R$ if  $\delta
\sigma = q\sigma \delta$.

We begin with the following key
observation which will allow us to erase skew derivations for certain Ore extensions.

\begin{lemma}
\label{erasing derivation}

  Let $\delta$ be a  $q$-skew $\sigma$-derivation of a finite
dimensional central simple algebra $Q$, where $q\ne 1$. Then there exists an  element $y=x-b\in Q[x;\si,\de]$, for a suitably chosen $b\in Q$, such that
$Q[x;\sigma,\delta]=Q[y;\sigma]$.
\end{lemma}
\begin{proof}

Assume  that  $\si$ is the identity on $Z(Q)$.  Then, by the
Skolem-Noether's theorem,  $\si$ is an inner automorphism of $Q$.
Let $a\in Q$ induce  $\si$, i.e. $\si(r)=ara^{-1}$, for all $r\in
Q$. One can easily check that $a^{-1}\de$ is a derivation of $Q$
and  $Q[x;\si,\de]=Q[a^{-1}x;a^{-1}\de]$. Let $z\in Z(Q)$. The
equalities  $\de(az)=\de(za)$ and $\si(z)=z$ easily
 lead to $\de(z)a=a\de(z)$, i.e.
$\de(z)=\si\de(z)$. Since $\si(z)=z$, we also have
$\de(z)=\de\si(z)=q\si\de(z)$. Hence $(1-q)\si\de(z)=0$, for $z\in
Z(Q)$.  By assumption $q\ne 1$, so the last equality implies
that the derivation $a^{-1}\de$ is zero on $Z(Q)$. Thus, by the
Skolem-Noether's theorem, $a^{-1}\de$ is an inner derivation of
$Q$.  Let the element $v\in Q$ be such that $a^{-1}\de(r)=vr-rv$,
for all $r\in Q$.  Using the above we have
$Q[x;\si,\de]=Q[a^{-1}x;a^{-1}\de]=Q[a^{-1}x-v]=Q[x-av;\si]$. This
gives the desired conclusion in the case $\si
|_{Z(Q)}=\mbox{id}_{Z(Q)}$.

Assume now that there exists  $c\in Z(Q)$ such that $u=\si(c)-
c\ne 0$. Then a classical and easy computation shows that
$Q[x;\si,\de]=Q[x+u^{-1}\de (c);\si]$. This completes  the proof
of the lemma.
\end{proof}

 For a prime right Goldie ring $R$, we
denote by $Q(R)$ the  classical right quotient ring of $R$.  Let us
recall that if $R$ is a prime PI ring then,  by  Theorems of Posner   and Kaplansky, $R$
is right Goldie, $Q(R)$ is a
central localization of $R$ and $Q(R)$ is a finite dimensional
central simple algebra (Cf. \cite{R}).

In the following lemma we collect known results which will be used later on.

\begin{lemma}
\label{isomorphisms and PI degree} Let $\si, \de$ be an automorphism and a
$\si$-derivation of a ring $R$, respectively. Then:
\begin{enumerate}
  \item[(1)] Let
$\mathcal{S}\subseteq R$ be a right Ore set of regular elements of $R$ such that
$\si(\mathcal{S})\subseteq \mathcal{S}$. Then:
\begin{enumerate}
 \item $\si$ and $\de$ have unique extensions to an automorphism and a $\si$-derivation
 of  $R\mathcal{S}^{-1}$, respectively. Moreover if $\de$ is a $q$-skew $\si$-derivation
 of $R$, then the extension  of $\de$ to $R\mathcal{S}^{-1}$ is also a $q$-skew
 $\si$-derivation of $\mathcal{R}S^{-1}$;
   \item
$\mathcal{S}$ is a right Ore set of regular elements of $R[x;\si,\de]$ and the localization
$(R[x;\si,\de])\mathcal{S}^{-1} $ is isomorphic to $ (R\mathcal{S}^{-1})[x;\si,\de]$.
\end{enumerate}

  \item[(2)]  Suppose $R$ is a prime \mbox{\rm PI} ring. Then:
  \begin{enumerate}
  \item $R[x;\si,\de]$ is a \mbox{\rm PI} ring if and only if $Q(R)[x;\si,\de]$ is a
  \mbox{\rm PI} ring;
      \item If $R[x;\si,\de]$ is a \mbox{\rm PI} ring, then
      $Q(R)[x;\si,\de]\subseteq Q(R[x;\si,\de]) $, and $Q(\Tn)$ is isomorphic to
      $ Q(\T{Q(R)})$.

  \end{enumerate}

\end{enumerate}

\end{lemma}
\begin{proof}
The statement (1) is part of folklore and we refer to Lemmas 1.3 and 1.4 of
\cite{G} for its proof.

The statement (2)(a) is exactly Proposition 1.6 of \cite{LM}.

(2)(b) Suppose $\Tn$ is a  PI ring. Notice that $\Tn$ is a prime ring as $R$ is such and the
statement (2)(a) implies that all rings appearing in (2)(b) are prime right Goldie rings.
Moreover every regular element of $R$ is regular in $\Tn$, as $\si$ is an automorphism
of $R$.   Now the thesis  is a direct consequence of (1) and universal properties of
localizations.
\end{proof}

\begin{remark} \label{remark generalities} It is known (Cf. \cite{M})
that if $R$ is a
prime (semiprime) right Goldie ring, then so is $\Tn$. Thus all rings appearing in
Lemma \ref{isomorphisms and PI degree}(2)(b) are prime (semiprime) Goldie. This
implies that $\T{Q(R)}\subseteq Q(\Tn)\simeq Q(\T{Q(R)})$, for
general Ore extension over prime (semiprime) Goldie rings.
\end{remark}

\begin{theorem} \label{lengh one thm}
Suppose that $\de$ is a $q$-skew $\si$-derivation of a prime ring $R$, where $q\ne
1$.  Then:
\begin{enumerate}
  \item[(1)] The following conditions are equivalent:
\begin{enumerate}
\item[(a)]  $\Tn$ is a $\mbox{\rm PI}$ ring;
\item[(b)] $\A{R}$ is a $\mbox{\rm
  PI}$ ring;
\item[(c)] $R$ is a PI ring and the restriction of $\si$ to the center $Z(R)$ of $R$ is an
  automorphism of finite order.
\end{enumerate}

  \item[(2)]  Suppose that  one of the equivalent conditions of (1) holds, then:
     \begin{itemize}
     \item[(a)] The quotient rings $Q(\Tn)$ and $Q(\A{R})$ are isomorphic. In particular, the $\mbox{\rm PI}$-degrees of  $\Tn$ and $\A{R}$ are equal;
      \item[(b)] there
      exists an  element $y\in \Tn$ such that $R[y;\si]\subseteq \Tn$.
      \end{itemize}
\end{enumerate}
\begin{proof}
(1) Suppose that one of the rings $R[x;\si]$ or $R[x;\si,\de]$ is PI.  Then $R$ is prime
PI
and Lemma \ref{isomorphisms and PI degree}(2)(a) shows that, replacing $R$ by $Q(R)$, we
may assume that $R$ is a central simple algebra finite dimensional over its center.
In this case, Lemma \ref{erasing derivation}
 implies  that $\A{R}\simeq \T{R}$, this gives the   equivalence between statements $(a)$ and $(b)$.

 The
 equivalence $(b)\Leftrightarrow (c)$ is part of Proposition $2.5$ in \cite{LM}.

 (2)(a) By assumption both $\Tn$ and $R[x;\si]$ are PI rings.  Then, by Lemma
 \ref{isomorphisms and PI degree}(2)(b) and Lemma \ref{erasing derivation}, we have: $Q(\Tn) \simeq Q(Q(R)[x;\si,\de])\simeq Q(Q(R)[x;\si])\simeq Q(R[x;\si])$.
This gives the first statement.

 Let us recall that the PI degree of a prime PI ring is also
the PI degree of its classical ring of quotients. This observation completes the proof
of (2).

 (2)(b) By assumption $\Tn$ is a PI ring, then so is
$Q(R)$  and Lemma \ref{erasing derivation} shows that there exists $b\in Q(R)$ such that
$Q(R)[x;\si,\de]=Q(R)[y';\si]$, where $y'=x-b$. By the theorem of Posner, we  can
pick  $0\ne c \in Z(R)$ such that
$y:=cy'=cx -a\in \Tn$. Thus, as $y'r=\si(r)y'$, we also have  $yr=\si(r)y$, for all
$r\in R$.
 Notice that
$y^n=c\si(c)\ldots\si^{n-1}(c)y'^n$. Therefore the set $\{y^n\}_{n=0}^\infty$, is left
$R$-independent.  This shows that  $R[y;\si]\subseteq \Tn$ and completes the proof.

\end{proof}
\end{theorem}

Since an Ore extension of a prime ring is again a prime ring, we
can use the above Theorem \ref{lengh one thm} to
erase derivations appearing in a PI iterated Ore extension as long
as the derivations are quantized. This erasing of derivations
process under the assumption of PI doesn't assume that the
derivations are locally nilpotent nor that the prime base ring is of
zero characteristic.  In this sense it gives a generalization of
the well known Cauchon's process of erasing derivations (Cf.
\cite{C}).  However some care is needed.
  Notice that even in the case of iterated Ore extension
  $R[x_1;\si_1.\de_1][x_2;\si_2.\de_2]$ of length  two, the iterated extension
$R[x_1;\si_1][x_2;\si_2]$ has no meaning in general, as the automorphism $\si_2$
is defined on the ring $R[x_1;\si_1.\de_1]$ but not on $R[x_1;\si_1]$.
The following lemma contains observations which both
explain the assumptions made in Theorem \ref{Iterated Ore extensions}
   and show that the iterated Ore extensions of automorphism type in this
theorem do exist.  First let us
consider a simple example.  Let $k$ be a field and $W=k\langle x_1,x_2 \;\vert \; x_2x_1=\lambda x_1x_2
 \rangle$ where $0\ne \lambda \in k$.   Notice that $W$ can be presented either as $k[x_1][x_2;\si]$ or $k[x_2][x_1;\tau]$, where $\si$ and $\tau$ are $k$-automorphisms  of appropriate  polynomial rings defined by $\si(x_1)=\lambda x_1$, and $\tau(x_2)=\lambda^{-1} x_2$.  The statement (2) of the following lemma generalizes this observation to Ore extensions.
\begin{lemma}
\label{basic remarks two indeterminates}

Suppose that   $\si_1,\si_2$ are automorphisms and  $\de_1$ is a
$\si_1$-derivation of  a ring $R$.  Let   $\lambda\in R$ be an
invertible element. Then:
\begin{enumerate}
\item[(1)] $\si_2$ can be extended to an automorphism of $R[x_1;\si_1,\de_1]$
      by setting $\si_2(x_1)=\lambda x_1$  if and only if
      $\si_2\si_1(r)=\lambda\si_1\si_2(r)  \lambda ^{-1}$ and $
      \si_2\de_1(r) = \lambda \de_1\si_2(r)$, for any $r\in
      R$;
\item[(2)] Suppose $\si_2$ has been extended to $R[x_1;\si_1,\de_1]$ as
above.  Then there exist an automorphism $\si'_1$ and a
$\si'_1$-derivation $\de'_1$ of $R[x_2;\si_2]$ such that
$R[x_1;\si_1,\de_1][x_2;\si_2]=R[x_2;\si_2][x_1;\si'_1,\de'_1]$
where
 $\si'_1|_{R}=\si_1$, $\si'_1(x_2)=\lambda^{-1}x_2$ and
 $ \de'_1|_R=\de_1$, $ \de'_1(x_2)=0$;
\item[(3)] With the same notation as in (2) above, suppose additionally that
$\de_1$ is a $q$-quantized $\si_1$-derivation. Then $\de'_1$ is
also a $q$-quantized $\si'_1$-derivation of $R[x_2;\si_2]$ if and only if
$\de_1(\lambda)=0$.
\end{enumerate}
\end{lemma}
\begin{proof}
A standard proof is  left to the reader.
\end{proof}

Let us make some remarks about the hypotheses that will appear in
the next theorem. In view of Lemma \ref{basic remarks two
indeterminates}(1) it will be natural, to avoid technicalities, to
assume that $\lambda_{ij}$'s from Theorem \ref{Iterated
Ore extensions} are central in $R_n$.  This means, in particular,
that the $\lambda_{ij}$'s are also fixed by appearing
automorphisms $\si_i$'s.  This, in turn, guarantees the
existence of the extension $T_n$ from the theorem.

The method of switching indeterminates in the proof of Theorem
\ref{Iterated Ore extensions}, based on Lemma \ref{basic remarks two
indeterminates}(2), goes back to the paper \cite{H}. The statement
(3) from the above lemma explains that, in the
process of switching indeterminates  in Theorem \ref{Iterated Ore
extensions}, we need all the $\lambda_{ij}$'s to be also invariant
under the action of all suitable skew derivations. A more
general situation will be considered in Proposition \ref{graded
method}.

\begin{theorem}
\label{Iterated Ore extensions}   Suppose   $R=R_0$ is a prime   ring   and $n\geq 1$.
Let  $R_i:=R_{i-1}[x_i;\si_i,\de_i]$,  $1\le i
\le n$, be a sequence of Ore extensions such that:

\begin{enumerate}
 \item[(i)] For any $1\le i\le n$ and $1\le j<i\le n$:\\
  $\bullet$ $\si_i|_{R_0}$ is an automorphism of $R_0$;\\
  $\bullet$  $\si_i(x_j)=\lambda_{ij}x_j$ where $\lambda_{ij}\in Z(R_0)$
 is a unit such that
 $\si_k(\lambda_{ij})=\lambda_{ij}$,
 $\de_k(\lambda_{ij})=0$, for all $i\le k\le n$;
 \item[(ii)] For any $1\le i \le n$, $\de_i$ is a $q_i$-skew
 $\si_i$-derivation of
 $R_{i-1}$, where $q_i\ne 1$.
 \end{enumerate}
Then the following conditions are equivalent:

\begin{enumerate}
\item[(1)] $R_n$ is a PI ring.
\item[(2)] There exist elements $y_i\in R_n$, where $1\le i \le n$, such that $R_n$
contains a subring $T_n$, which is an Ore extension of the form
$T_n:=R_0[y_1;\tau_1][y_2;\tau_2]\dots [y_n;\tau_n]$ where
$\tau_i|_{R_0}=\si_i|_{R_0}$,  for $1\le i \le n$ and
$\tau_i(y_j)=\lambda_{ij}y_j$, for $1\le j < i \le n$.  Moreover    $T_n$ is
\mbox{\rm PI} and the quotient rings $Q(T_n)$ and $ Q(R_n)$ are isomorphic.
\end{enumerate}
\end{theorem}

\begin{proof}
Notice first that  all rings considered in the theorem are prime, as $R=R_0$ is
prime.
Observe also that    $\si_i(R_j)=R_j$,  for $0\le j < i \le n$.

(1) $\Rightarrow$ (2)  We proceed by induction on $n$. The case
$n=1$ is a direct consequence of Theorem \ref{lengh one thm}.

Suppose now that $n>1$ and $R_n$ is a PI ring.  Theorem \ref{lengh one thm}
shows that there exists an element $y_n\in R_n$ such that the subring  $T=R_{n-1}[y_n;
\si_n]\subseteq R_n $ generated by $R_{n-1}$ and $y_n$ satisfies $Q(T)\simeq Q(R_n)$.

\noindent Using Lemma \ref{basic remarks two indeterminates}(2)
 we can write
$T=R_{n-2}[y_n;\si_n][x_{n-1};\si_{n-1}',\de_{n-1}']$, where $\si'_{n-1}$ is given by
$\si'_{n-1}|_{R_{n-2}}=\si_{n-1}|_{R_{n-2}}$,
$\si_{n-1}'(y_n)=\lambda_{n, n-1}^{-1}y_n$ and $\de_{n-1}'$ is the extension of
$\de_{n-1}|_{R_{n-2}}$ to $R_{n-2}[y_n;\si_n]$ obtained by setting
$\de_{n-1}'(y_n)=0$.  Continuing this process we can present $T$ in the following way:

$$T=R_0[y_n;\si_n][x_1;\si_1',\de_1']\dots
[x_{n-1};\si_{n-1}',\de_{n-1}'], $$
 where   $\si_{i}'(y_n)=\lambda_{ni}^{-1}y_n$ and
$\si'_i(x_j)=\lambda_{ij}x_j$,  for all $1\le i \le n-1$ and $1\le
j< i\le n-1$.  Lemma \ref{basic remarks two indeterminates}(3)
shows that the $\si'_i$-derivations $\de'_i$ remain
$q_i$-quantized, $ 1 \leq i\leq {n-1}$.   We can now apply the
induction hypothesis to $T$, replacing $R_0$ by $R_0[y_n;\si_n]$,
and conclude that there exist elements $y_1,\dots, y_{n-1}\in T $ such that the subring $T_n$
generated by $R_0[y_n;\si_n]$ and $y_1,\dots, y_{n-1}$ satisfies
$$
T_n:=R_0[y_n;\si_n][y_1;\si''_1]\dots [y_{n-1};\si''_{n-1}]
\subseteq T \;\mbox{\rm and}\; Q(T_n) \simeq Q(T)
$$
where $ \si''_i|_{R_0}=\si'_i|_{R_0}=\si_i|_{R_0}$,
$\si''_i(y_n)=\si'_i(y_n)=\lambda_{ni}^{-1}y_n$, $
\si''_i(y_j)=\si_i(y_j)=\lambda_{ij}y_j$,
 for all $1\le i \le n-1 $ and $1\le
j<i\le n-1 $.

\noindent Now, using Lemma \ref{basic remarks two
indeterminates}(2) again, we can reorder the $y_i's$  to get that
$T_n=R_0[y_1;\tau_1][y_2;\tau_2]\dots[y_n;\tau_n]$, where
$\tau_i|_{R_0}=\si_i|_{R_0}$ and $\tau_i(y_j)=\lambda_{ij}y_j$,
for any $1\le j< i\le n$.   Therefore
$$
R_0[y_1;\tau_1][y_2;\tau_2]\dots[y_n;\tau_n]=T_n\subseteq T
\subseteq R_n
$$
as required.
As it was proved above we also have $Q(R_n)\simeq Q(T)\simeq Q(T_n)$.
Clearly $T_n$ is a PI ring as a subring of the PI
ring $R_n$.  This completes the proof of $(1)\Rightarrow (2)$.

The implication $(2) \Rightarrow (1)$    is clear.
\end{proof}

The above theorem is a partial generalization of the main result of
\cite{H}.  Comparing Theorem \ref{Iterated Ore extensions} above
with Theorem 4.6 \cite{H}, observe that  in \cite{H} it is additionally assumed that each $\de_i$,
with $1\leq i \leq n$, extends to  locally nilpotent iterative
higher $q_i$-skew $\si_i$-derivation on $R_{i-1}$ (see \cite{H}
for details).  Notice also that in \cite{H}, $R$ is a noetherian
domain which is an algebra over a field $k$ and   $
q_i, \lambda_{ij}\in k  $, where $q_i\not\in\{0,1\}$. This is due to the fact that
higher $q$-skew derivations
where used in \cite{H} for erasing derivations from Ore
extensions. The assumption that $\si_i|_{R_0}$ is an automorphism of $R_0$,
for all  $1\le i\le n$, from Theorem \ref{Iterated Ore extensions}  was not
formally stated in Theorem 4.6 \cite{H} but it was used in its proof.

Most of the quantum algebras can be presented as iterated Ore
extensions of the form $k[x_1][x_2;\si_2,\de_2]\ldots [x_n;\si_n,\de_n]$, where $k$
is a field and the appearing automorphisms and skew derivations are
as in Theorem \ref{Iterated Ore extensions}. Thus we record the
following theorem which covers this case  by taking  $R_0=k$,
$\si_1=\mbox{\rm id}_k,\de_1=0$ and $q_1=0$.

To make the presentation a bit shorter, we formulate the theorem
in the language of  algebras. If the Ore extension $\Tn$ is PI,
then so is the ring $R$, thus we will assume that $R$ satisfies a
polynomial identity.

\begin{theorem}
\label{Haynal}
Suppose that  $R=R_0$ is a prime \mbox{\rm PI} algebra over a field $k$ and $n\geq 1$.
Let  $R_i:=R_{i-1}[x_i;\si_i,\de_i]$,  $1\le i
\le n$, be a sequence of Ore extensions such that each $\si_i$  is a
$k$-linear automorphism  of $R_{i-1}$  and each  $\de_i$ is a $k$-linear $\si_i$-derivation of $R_{i-1}$ such that:
\begin{enumerate}
\item[(i)]  $\si_i|_{R_0}$ is an automorphism of $R_0$ of finite order, for any $1\le i \le n$;
\item[(ii)] $\si_i(x_j)=\lambda_{ij}x_j$ where $0\ne\lambda_{ij}\in
k$, for any  $1\le j < i\le n$;
\item[(iii)]  $ \de_i$ is a $q_i$-skew
$\si_i$-derivation of $R_{i-1}$, where $1\ne q_i\in k$, for  any $1\le i \le n$.
\end{enumerate}
Then the following conditions are equivalent:
\begin{enumerate}
\item[(1)] $R_n$ is a \mbox{\rm PI} algebra;
\item[(2)] $T_n=R_0[y_1;\si'_1][y_2;\si '_2]\cdots[y_n;\si '_n]$ is a \mbox{\rm PI}
algebra  where $\si'_i|_{R_0}=\si_i|_{R_0}$ and $\si'_i(y_j)=\lambda_{ij}y_j$,
for $1\le i\le n$  and $1\le j<i\le n$;
\item[(3)]  $\lambda_{ij}$ is a root of unity, for any $1\le j<i\le n$;
\item[(4)]$\si_i$ is an automorphism of finite order of
$R_{i-1}$, for any $1\le i\le n$.
\end{enumerate}
Moreover if one of the above equivalent conditions holds, then the algebras $R_n$ and $T_n$
have isomorphic classical rings of quotients  and  equal  \mbox{\rm PI} degrees.
\end{theorem}
\begin{proof} The implication
$(1)\Rightarrow (2)$ and the additional statements are direct consequences of   Theorem \ref{Iterated Ore extensions}.

$(2)\Rightarrow (3)$ Let us fix $ 1\le j<i \le n$. Since $\si'_j|_{R_0}=\si_j|_{R_0}$, the assumption  $(i)$ gives an  $0\ne n_j\in
\mathbb N$ such that $(\si_j'|_{R_0})^{n_j}=\mbox{\rm id}_{R_0}$. This means that
$y_j^{n_j}\in Z(R_0[y_j;\si'_j])$.

Now let us consider the subring $T_{ji}$ of $T_n$ generated by
$R_0, y_j,y_i$. The assumptions imposed imply that $T_{ji}$ is the
Ore extension $R_0[y_j;\si'_j][y_i;\si'_i]$ which is PI as a
subalgebra of $T_n$.  Hence, by Theorem \ref{lengh one thm}(1), $\si'_i$ is of finite order on
$Z(R_0[y_i;\si'_i])$. In particular there exists $0\ne l_i\in
\mathbb N$ such that $\si_i'^{l_i}(y_j^{n_j})=y_j^{n_j}$.  This
leads to $\lambda _{ij}^{n_jl_i}=1$, as required.

$(3)\Rightarrow (4)$ Fix $1 \le i \le n$.  By assumption, we can find $m\geq 1$ such that  $(\si_i|_{R_0})^m=\mbox{\rm id}_{R_0}$  and $\lambda_{ij}^m=1$, for all  $1\le
j<i\le n$.  This implies that the automorphism $\si_i$ is of finite order on $R_{i-1}$.

$(4)\Rightarrow (1)$   We will proceed by induction
on $n$. For $n=1$, the implication is a direct consequence of  Theorem \ref{lengh one thm}(1).

 Suppose $n> 1$.  Recall that, by assumption, $R_0$ is a prime PI algebra.  The case $n=1$,
 treated above, shows that $R_1= R_0[x_1;\si_1,\de_1]$ is a prime
 PI algebra as well.  Thus, the set $\mathcal{S}$ of all regular elements of $R_1$ is a right
 Ore set and, using Lemma \ref{isomorphisms and PI degree}(1)(b) repeatedly, we see that $\mathcal{S}$ is a right
 Ore set of regular elements of $R_n$ and $R_n\mathcal{S}^{-1}$  is
 isomorphic to
 $Q(R_1)[x_2;\si_2,\de_2]\ldots[x_n;\si_n,\de_n]$.

Notice that, due to the assumption (4), the extensions
 of the automorphisms  $\si_i$ to $Q(R_1)$,    $2\leq i\leq n$,        possess   the property (i)  with respect to $Q(R_1)$    and clearly also satisfy (ii).  Moreover   Lemma \ref{isomorphisms and PI degree}(1)(a) implies that the extensions of  $\si_i$-derivations $\de_i$ to $Q(R_1)$ remain $q_i$-quantized, $2\leq i\leq n$, i.e. (ii) holds. Therefore,  applying  the induction hypothesis  to
  $Q(R_1)[x_2;\si_2,\de_2]\ldots[x_n;\si_n,\de_n]\simeq R_n\mathcal{S}^{-1}$ we can conclude that $ R_n\mathcal{S}^{-1}$ satisfies polynomial identity and  so does  $R_n$.
  \end{proof}
 Both Theorem 1.2 (1) and  Corollary 4.7 of \cite{H} are direct consequences of the above
 theorem.  Moreover we relaxed the assumptions from \cite{H} that $\de_i$'s have to extend
 to locally nilpotent iterative
higher $q_i$-skew $\si_i$-derivations and that $R$ is a noetherian domain.

Let us mention that Haynal applied
  Corollary 4.7 \cite{H} and a result of De Concini and Procesi (Cf.
\cite{CC})
to   compute   PI degrees of some quantum
algebras.

Notice that the    Weyl algebra over a field $k$ of characteristic 0  does not satisfy a polynomial
 identity but $k[x]$ does. This shows that the implication  $(2)\Rightarrow (1)$   of Theorem \ref{Haynal} does not hold if we allow one of the $q_i$'s to be equal to one.
On the other hand, we will show  in Theorem \ref{graded method},  that the implication $(1)\Rightarrow (2)$ of Theorem  \ref{Haynal}  holds even when skew derivations $\de_i$'s are not quantized and   the base ring $R$ is not prime. The example of
Weyl algebra over a field $k$ of characteristic $p\ne 0$ shows that the PI degrees of $R_n$ and $T_n$ can be different in this case.

For the proof of Theorem \ref{graded method}  some preparation is needed.
Let $\si$ be an automorphism of the $\mathbb N$-graded ring $B=\bigoplus_{i =
0}^\infty  B_i$ (resp. of the filtered ring $C=\bigcup _{i=0}^\infty
C_i$), we say that $\si$ respects the gradation (resp. the
filtration) if $\si (B_i)=B_i$ (resp. $\si (C_i)=C_i$), for all $i\ge 0$.
  The associated graded ring $C_0\oplus \bigoplus_{i =
1}^\infty (C_i/C_{i-1})$ of the filtered ring $C=\bigcup _{i=0}^\infty
C_i$ will be denoted by  $ \mathfrak{gr}(C)$.

Henceforward we will extend slightly our previous notation and write also $A[x;\si]$
for the additive group of all polynomials from skew polynomial ring $R[x;\si]$
consisting of all polynomials with coefficients in an additive group $A$ of
$R$.

\begin{lemma}\label{graded}
 Suppose that $C=\bigcup _{i=0}^\infty C_i$ is a filtered ring and
 $\si$ is an automorphism of $C$ which respects the filtration.  Then:
\begin{enumerate}
\item[(1)] $\si$ induces an automorphism $\overline{\si}  $
 of the associated graded ring $ \mathfrak{gr}(C)$ which respects the gradation of $\mathfrak{gr}(C)$. Moreover $\overline{\si} |_{C_0}=\si|_{C_0}$.
\item[(2)]  The ring $B=C[x;\si]$ has a natural filtration
$B=\bigcup_{i=0}^{\infty} B_i$, where
$B_i=C_i[x;\si]$.  The associated graded ring $\mathfrak{gr}(B)$ is isomorphic
to $\mathfrak{gr}(C)[x;\overline{\si}]$,
where $\overline{\si}$ denotes the automorphism defined in (1).
\item[(3)] Let $T_n=C[x_1;\si_1]\dots [x_n;\si_n]$ be an iterated Ore extension such that:\\
$\bullet$   For any $1\leq i\leq n$, $\si_i|_C$ respects the filtration of $C$;\\
$\bullet$    For any $1\leq j<i\leq n$,  $\si_i(x_j)=a_{ij}x_j$,
for some invertible elements $a_{ij}\in C_0$.  \\
Then $T_n=\bigcup_{i=0}^\infty C_i[x_1;\si_1]\cdots[x_n;\si_n]$ is a filtration of $T_n$ such that
the associated graded ring $\mathfrak{gr}(T_n)$ is isomorphic to  $\mathfrak{gr}(C)[x_1; \si _1']\dots[x_n; \si_n']$ where the automorphisms $\si_i'$'s satisfy:\\
$\bullet$ For any  $1\leq i\leq n$, $\si_i'|_{\mathfrak{gr}(C)}=\overline{\si_i|_{C}}$ defined as in (1);\\
$\bullet$    For any $1\leq j<i\leq n$,  $\si_i'(x_j)=a_{ij}x_j$.

\end{enumerate}
\end{lemma}
\begin{proof}
(1)  Let us extend the automorphism $\si$ of $C$ to an automorphism of $C[x]$
 by setting $\si(x)=x$. Then, by the assumption, $\si$ preserves the filtration
 of $C$.  Hence $\si$ induces an automorphism of the Rees extension
 $\mathfrak{R}(C)=\sum_{i=0}^\infty C_ix^i\subseteq C[x]$. Let  $I=\mathfrak{R}(C)x$.
 Then $\si(I)=I$ and $\si$
 induces an automorphism $\overline \si$ of $\mathfrak{R}(C)/I\simeq \mathfrak{gr}(C)$ which has the desired properties.

 (2) By the statement $(1)$, $\mathfrak{gr}(C)[x;\overline{\si}]$ exists.  The automorphism  $\si$ preserves the filtration of the base ring $C$.  Thus it is clear that $B=\bigcup_{i=0}^\infty B_i$,
 where $B_i=C_i[x;\si]$ is a filtration of
 $C[x;\si]$.  The isomorphism between the graded rings
 $\mathfrak{gr}(C[x;\si])$ and $\mathfrak{gr}(C)[x;\overline{\si}]$ is given by the natural isomorphism of homogeneous components $C_i[x;\si]/C_{i-1}[x;\si]$ and $(C_i/C_{i-1})[x;\overline{\si}]$, where  $i\geq 0$ and $C_{-1}=0$.

 $(3)$ We proceed by induction on $n$. The case $n=1$ is given by the statement
  $(2)$.

 Let $n>1$. By the induction hypothesis,  we know  that
$T_{n-1}=\bigcup_{i=0}^\infty C_i[x_1;\si_1]\cdots[x_{n-1};\si_{n-1}]$
is a filtration for $T_{n-1}$ such that:\\
$\bullet$ the  associated graded ring $\mathfrak{gr}(T_{n-1})$  is isomorphic to
$\mathfrak{gr}(C)[x_1;\si_1']\cdots[x_{n-1};\si_{n-1}']$
\\where \\
$\bullet$  $\si_i'|_{\mathfrak{gr}(C)}=\overline{\si_i|_{C}}$ defined as in (1), for any  $1\leq i\leq n-1$;\\
$\bullet$     $\si_i'(x_j)=a_{ij}x_j$, for any $1\leq j<i\leq n-1$.\\
Let us write $T_n$ as  $T_n=T_{n-1}[x_n;\si_n]$. Notice that, by the assumptions imposed on $\si_n$,  the filtration of $T_{n-1}$ is  respected by $\si_n$. Therefore, by (2), we can extend the filtration of $T_{n-1}$ to $T_n$ by setting
$$(T_n)_i=(T_{n-1})_i[x_n;\si_n]=C_i[x_1;\si_1]\ldots[x_n;\si_n],\;\mbox{for all}\;\; i\geq 0.$$

Hence, making use of (2) and (1) with $C=T_{n-1}$, we obtain $$\mathfrak{gr}(T_n)\simeq\mathfrak{gr}
(T_{n-1})[x_n;\si_n']\simeq \mathfrak{gr}(C)[x_1;\si_1']\cdots[x_{n-1};\si_{n-1}'][x_n;\si_n']$$ where $\si_n'=\overline{\si_n}$, as defined in (1).
In particular, $\overline{\si_n}\vert_{C_0[x_1;\si_1]\ldots [x_{n-1};\si_{n-1}]}=\si_n\vert_{C_0[x_1;\si_1]\ldots [x_{n-1};\si_{n-1}]}$. This shows that $ \si_n'(x_i)=\si_n(x_i)=a_{ni}x_i$, for
$1\le i\le n-1$.  Since $\mathfrak{gr}(C)\subseteq \mathfrak{gr}(T_{n-1})$, we also have $\si_n'\vert_{\mathfrak{gr}(C)}=\overline{\si_n}\vert_{\mathfrak{gr}(C)}=\overline{\si_n\vert_C}$.   This shows that the automorphism $\si_n'$ has the desired properties and completes the proof of the   lemma.
\end{proof}
\begin{lemma}\label{extending filtrations}
 Let $n\geq 1$ and $R=R_0$.  Suppose that, for $1\le i\le n$, the
iterated Ore extension $R_i=R_{i-1}[x_i;\si_i,\de_i]$ is given
and:
\begin{enumerate}
\item[(i)]  for any $1\leq  i\leq n$, $\si_i|_{R_0}$ is an automorphism of $R_0$;
 \item[(ii)] for any $1\leq j<i\leq n$, $\si_i(x_{j})=a_{ij}x_j+c_{ij}$, where
 $a_{ij}\in  R_0$ is invertible and $ c_{ij}\in R_{j-1}$.
\end{enumerate}
Then $\si_i \vert_{R_j}$ is an
automorphism of $R_j$, for any $0\le j < i \le n$. Moreover   $\si_i \vert_{R_j}$  respects the filtration of $R_j$ determined by the degree in
$x_j$, when $j\geq 1$.
\end{lemma}
\begin{proof}
 The inclusion $\si_i(R_j)\subseteq R_j$ is clear from the definition of $\si_i$.
The reverse inclusion is obtained by induction on $0\le j<i$.  The case $j=0$
is given by the hypothesis $(i)$.  If $j>0$, the induction hypothesis shows that there
exist elements $b_{ij}\in R_0$, $d_{ij}\in R_{j-1}$ such that
$\si_i(b_{ij})=a_{ij}^{-1}$ and $\si_i(d_{ij})=c_{ij}$.  This gives
$x_j=\si_i(b_{ij}x_j-d_{ij})$.  In particular, $x_j\in \si_i(R_j)$.  It is now
easy to conclude that $\si_i\vert_{R_j}$ is an automorphism of $R_j$ which
respects the filtration of $R_j$ given by the degree in $x_j$, when $j\geq 1$.
\end{proof}
 It was observed in Proposition 1.2 \cite{LM} that if a filtered ring $C$
satisfies a polynomial identity, then the same is true for the associated graded
ring $\mathfrak{gr}(C)$. Thus the skew polynomial
ring $\A{R}$ is always a PI ring, provided $\T{R}$ is such.
In the following theorem we extend this result to iterated Ore extensions.

\begin{theorem}\label{graded method}
Let $n\geq 1$ and $R=R_0$.  Suppose that, for $1\le i\le n$, the
iterated Ore extension $R_i=R_{i-1}[x_i;\si_i,\de_i]$ is given
and:
\begin{enumerate}
\item[(i)]  for any $1\leq  i\leq n$, $\si_i|_{R_0}$ is an automorphism of $R_0$;
 \item[(ii)] for any $1\leq j<i\leq n$, $\si_i(x_{j})=a_{ij}x_j+c_{ij}$, where
 $a_{ij}\in  R_0$ is invertible and $ c_{ij}\in R_{j-1}$.
\end{enumerate}
 Then there exists an iterated Ore extension
 $T_n=R_0[y_{1};\si'_{1}]\ldots [y_n;\si_n']$  such that:
 \begin{enumerate}
   \item[(1)] $\si'_1=\si_1$, $\si'_i|_{R_0}=\si_i|_{R_0}$ and $\si'_i(y_j)=a_{ij}y_j$, for all $ 1\leq j< i\leq n$;
   \item[(2)] If $R_n$ is a PI ring, then $T_n$ is also a PI ring.
 \end{enumerate}
\end{theorem}
\begin{proof}
Let us first remark that, for any $0\le j < i \le n$, $\si_i \vert_{R_j}$ is an
automorphism of $R_j$ which respects the filtration of $R_j$ determined by the degree in
$x_j$.

$(1)$  We construct, by induction, a sequence $W_0,W_{1},\dots,
W_{n}=T_n$ of filtered rings. Let us put $W_0=R_n$.   $W_0$
 is naturally filtered by the degree
in $x_n$ and  we set $W_1=\mathfrak{gr}(W_0)\simeq R_{n-1}[y_{n};\si_n] $.  By Lemma \ref{extending filtrations}, the filtration of $R_{n-1}$ given by the degree in $x_{n-1}$ is respected by $\si_n$. The filtration on $W_1$ is given by   extending, as in Lemma \ref{graded}(2),  this  filtration of $R_{n-1}$.

 Suppose $1\leq s< n$ and  the extension $W_s=R_{n-s}[y_{n-s+1};\mu_{n-s+1}]\dots[y_n;\mu_n]$ is defined,
  where:
 $$ \mu_{i}|_{R_{n-s} }=\si_i |_{R_{n-s} }\;\mbox{and}\;\mu_i(y_j)=
 a_{ij}y_j ,\;\mbox{ for any}\;  n-s+1 \le i \le n\;\mbox{and}\; n-s+1\leq j<i.$$

Now we can apply  Lemma \ref{extending filtrations} and Lemma \ref{graded}     with $C=R_{n-s}=R_{n-s-1}[x_{n-s},\si_{n-s}, \de_{n-s}]$ filtered  by the degree in $x_{n-s}$ to obtain a filtration on $W_s$ such that
$$\mathfrak{gr}(W_s)\simeq \mathfrak{gr}(R_{n-s})[y_{n-s+1};\mu'_{n-s+1}]\dots[y_n;\mu'_n]\simeq R_{n-s-1}[y_{n-s};\si_{n-s}][y_{n-s+1};\mu'_{n-s+1}]
\dots[y_n;\mu'_n]$$
where
 $$ \mu'_{i}|_{R_{n-s-1} }=\si_i |_{R_{n-s-1} }\;\mbox{and}\;\mu'_i(y_j)=
 a_{ij}y_j ,\;\mbox{ for any}\;  n-s  \le i \le n\;\mbox{and}\;  n-s\leq  j<i$$
  with   $\mu'_{n-s}=\si_{n-s}$.
We  define
$W_{s+1}$ by setting $W_{s+1}=\mathfrak{gr}(W_s)$.
  The desired iterated skew polynomial ring $T_n$ is $T_n=W_n$.

(2)  Suppose that $R_n=W_0$ satisfies a polynomial identity. Notice that each $W_{s+1}$, with $0\leq s<n$, is defined as $W_{s+1}=\mathfrak{gr}(W_s)$. Therefore $T_n=W_n$ is a PI ring, by   Proposition 1.2 of \cite{LM}  which says  that the associated graded
ring $\mathfrak{gr}(C)$ of a filtered ring $C$ is PI, provided $C$ is such.
\end{proof}

\end{document}